\documentclass[12pt]{amsart}
\usepackage{amssymb, amscd, hyperref}

\def\vp{\varphi}
\def\t{\theta}
\def\<{\langle}
\def\>{\rangle}
\def\irr{\mathrm{Irr}}
\def\ibr{\mathrm{IBr}}
\def\s{\sigma}
\def\tr{\mathrm{tr}}
\def\C{\mathbb{C}}
\def\R{\mathbb{R}}
\def\Q{\mathbb{Q}}
\def\gal{\mathrm{Gal}}
\def\l{\lambda}

\newtheorem{Lemma}             {Lemma}

\newtheorem{Proposition}[Lemma]{Proposition}

\newtheorem{Example}    [Lemma]{Example}
\newtheorem{Definition}    [Lemma]{Definition}
\newtheorem{Theorem}    [Lemma]{Theorem}

\newenvironment{Proof}{\textit{{\textbf{Proof.}}}}{\hfill\ensuremath{\blacksquare}\medskip}

\title[Gallagher's Theorem for Real-valued Brauer Characters]{Gallagher's Theorem for Real-valued Brauer Characters}
\author{Daryl Zane Adriano}

\address{Department of Mathematics and Statistics\\Maynooth University\\IRELAND}
\email{daryl.adriano.2022@mumail.ie}

\date{\today}
\begin{document}
	\maketitle
	
	\section*{Abstract}
	Let $G$ be a finite group, let $N$ be a normal subgroup of $G$, and let $\t$ be an irreducible $p$-Brauer character of $N$, for a prime $p$. It is well known that the number of irreducible Brauer characters of $G$ lying over $\t$ equals the number of $p$-regular $\t$-good conjugacy classes of $G_\t/N$.	
	
	In this note we express the number of real-valued irreducible Brauer characters of $G$ lying over $\t$ in terms of the $p$-regular $\t$-good conjugacy classes of $G_\t/N$. This is a modular analogue of one of the main results in a recent paper \cite{Mur23} by John Murray.
	\section*{Introduction}
	Throughout this paper, $G$ is a finite group, $p$ is a prime, and $\t$ is an irreducible Brauer character of a normal subgroup $N$ of $G$. As usual $G_\t$ denotes the inertia group of $\t$ in $G$. Recall that an irreducible Brauer character $\vp$ of $G$ is said to lie over $\t$ if $\t$ is a constituent of $\vp_{N}$.
	
	We recall what it means for a $p$-regular conjugacy class of $G_\t/N$ to be $\t$-good (cf. \cite{Gal70}). For any $p$-regular $x \in G_\t$, $\t$ extends to an irreducible Brauer character $\t_x$ of $N\< x \>$. Define the group $C_{G_\t}(Nx):= \{ g \in G_\t| [g,x] \in N\}$. Then $Nx \in G_\t/N$ is said to be $\t$-good if $\t_x^c = \t_x$ for all $c \in C_{G_\t}(Nx)$. 
	
	It can be shown that this does not depend on the choice of extension $\t_x$. Also $Nx \in G_\t/N$ is $\t$-good if and only if all of its $(G_\t/N)$-conjugates are $\t$-good. If $Nx$ is $\t$-good, we call the $(G_\t/N)$-conjugacy class of $Nx$ $\t$-good as well. The following is a special case of one of the main results in \cite{Isa88}:
	\medskip
	
	\noindent\textbf{Isaacs' Theorem.} \label{IT}\textit{The number of irreducible Brauer characters of $G$ lying over $\t$ is equal to the number of $\t$-good $p$-regular conjugacy classes of $G_\t/N$.}
	\medskip
	
	If $p$ does not divide the order of $G$, this reduces to a theorem by P.X. Gallagher (cf. \cite{Gal70}). This states that if $\t \in \irr(N)$ is an ordinary irreducible character of $N$, the number of irreducible characters of $G$ lying over $\t$ is equal to the number of $\t$-good conjugacy classes of $G_\t/N$.
	
	In this note, we count the number of real-valued irreducible Brauer characters of $G$ lying over $\t$. Now $G$ has no such Brauer character unless $\t$ is $G$-conjugate to $\overline{\t}$. So we assume from now on that $\t$ is $G$-conjugate to $\overline{\t}$.
	
	Before stating our main result, we need to define an indicator $\s$ on the $p$-regular $\t$-good conjugacy classes of $G_\t/N$. To start, we define the extended inertia group of $\t$ in $G$ to be the subgroup \medskip
	$$G_\t^* := \{g \in G| \t^g = \t \ \text{or} \ \overline{\t} \}.$$
	
	So $G_\t \leq G_\t^*$ and $[G_\t^*: G_\t] \leq 2$. Moreover $G_\t = G_\t^*$ if and only if $\t = \overline{\t}$. Throughout this paper, we fix $e \in G_\t^*$ to be any element that satisfies $\t^e = \overline{\t}$, and for any $Nx \in G_\t/N$, $\text{Cl}(Nx)$ will just denote the conjugacy class of $G_\t/N$ containing $Nx$ unless otherwise stated.
	
	\begin{Definition}
		Let $Nx \in G_\t/N$ be $\t$-good and $p$-regular, and let $\t_x$ be an arbitrary extension of $\t$ to $N\< x \>$. We assign an indicator $\s(Nx) \in \{0,+1,-1\}$ as follows. Suppose first that $\mathrm{Cl}(Nx) = \mathrm{Cl}(N(x^{-1})^{e^{-1}})$. Then there exists $\varepsilon \in G_\t^*$ such that $\t^\varepsilon = \overline{\t}$ and $Nx^{\varepsilon^{-1}} = Nx^{-1}$. Set $\s(Nx) = +1$, if $\t_x^\varepsilon = \overline{\t}_x$, and $\s(Nx) = -1$, if $\t_x^\varepsilon \neq \overline{\t}_x$. If $\mathrm{Cl}(Nx) \neq \mathrm{Cl}(N(x^{-1})^{e^{-1}})$ then no such $\varepsilon$ exists, and we set $\s(Nx) = 0$.
	\end{Definition}
	
	 This is analogous to the definitions established in \cite[\text{Section }1]{Mur23}. 
	
	Since we are letting $Nx \in G_\t/N$ be $\t$-good, it is straightforward to show that $\s(Nx)$ does not depend on the choice of $\varepsilon$, nor on the choice of extension $\t_x$. Additionally, $\s(Nx) = \s(Nx^g)$ for any $g \in G_\t$. So $\s$ is constant on conjugacy classes of $G_\t/N$. Thus we can refer to each $\t$-good conjugacy class of $G_\t/N$ as being of $\s_+$, $\s_-$ or $\s_0$ type. Our main result is:
	\newtheorem{theorema}{Theorem}
	\setcounter{theorema}{0}
	\renewcommand*{\thetheorema}{\Alph{theorema}}
	
	\begin{Theorem} \label{MT}
		Let $N \trianglelefteq G$ and let $\t \in \ibr(N)$. Then the number of real-valued irreducible Brauer characters of $G$ lying over $\t$ is equal to the number of $p$-regular $\t$-good conjugacy classes of $G_\t/N$ of $\s_+$ type minus the number of $p$-regular $\t$-good conjugacy classes of $G_\t/N$ of $\s_-$ type.
	\end{Theorem}
	When $p$ does not divide the order of $G$, this reduces to \cite[Theorem A]{Mur23}.
	
	We use a generalized version of Brauer's Permutation Lemma to prove Theorem \ref{MT}. In Section $1$, we introduce an invertible matrix $M$ that emerges from the proof of Isaacs' Theorem in \cite{Isa88}. The rows of $M$ will correspond to the irreducible Brauer characters of $G_\t$ lying over $\t$, while its columns correspond to the $p$-regular $\t$-good conjugacy classes of $G_\t/N$.
	
	In Section $2$, we will define the matrix $\overline{M^e}$ using the involution $\vp \mapsto \overline{\vp^e}$ on $\ibr(G_\t | \t)$. We will show that there exists a permutation matrix $X$ such that $\overline{M^e} = XM$, while there exists a monomial matrix $Y$ such that $\overline{M^e} = MY$. Since $XM = MY$, we get $X = MYM^{-1}$, and hence $\tr X = \tr (MYM^{-1}) = \tr Y$. 
	
	So to prove Theorem \ref{MT}, we will show in Section $2$ that $\tr X$ is equal to the number of real-valued irreducible Brauer characters of $G$ lying over $\t$, while $\tr Y$ is equal to the number of $p$-regular $\t$-good conjugacy classes of $G_\t/N$ of $\s_+$ type minus the number of $p$-regular $\t$-good conjugacy classes of $G_\t/N$ of $\s_-$ type.
	
	We will then look at examples and special cases in Section $3$. Finally, in Section 4, we will mention a generalization of Theorem \ref{MT}.  
	\section{Gallagher's Theorem for Brauer characters}
	Given any finite group $H \supseteq N$, we let $\ibr(H|\t)$ denote the set of irreducible Brauer characters of $H$ that lie over $\t$.
	
	By the Clifford correspondence, $|\ibr(G|\t)| = |\ibr(G_\t|\t)|$ (cf. \cite[\text{Theorem 8.9}]{Nav98}). So we focus on the inertial group $G_\t$.
	
	Isaacs proved a more general version of Gallagher's Theorem in \cite{Isa88}. We sketch Isaacs's proof of his theorem, restricted to Brauer characters, and define the matrix $M$ we need afterwards.
	
	\medskip
	
	\textbf{\textit{Sketch proof of $\ref{IT}{\text{Isaacs' Theorem}}$.}} Suppose $G_\t/N$ has $t$ $p$-regular conjugacy classes. By \cite[\text{Corollary 5.5}]{Isa88}, there is a choice $\{ x_1, \dots, x_t\}$ of $p$-regular elements in $G_\t$, and a choice of extensions $\t_i \in \ibr(N\< x_i \>)$ of $\t$, such that the following hold:
	\begin{itemize}
		\item $\{ Nx_1, \dots, Nx_t\}$ forms a complete system of representatives for the $p$-regular $\t$-good conjugacy classes in $G_\t/N$.
		\item $\t_i(x_i) \neq 0$ for all $i \in \{ 1, \dots, t\}$.
	\end{itemize}
	We set $T:=\{ x_1, \dots, x_t\}$. Theorem 6.6 of \cite{Isa88} states that restriction to $T$ defines a $\C$-vector space isomorphism from the $\C$-vector space  $\mathrm{cf}(G_\t^0|\t)$ spanned by $\ibr(G_\t|\t)$ to the space of all complex functions on $T$. \hfill\ensuremath{\blacksquare}\medskip
	
	After labeling the elements of $\ibr(G_\t|\t)$ as $\{ \vp_1, \dots, \vp_t\}$, \cite[\text{Theorem 6.6}]{Isa88} tells us that the $t \times t$ matrix \begin{equation} \label{Mat}	
		M := [\vp_i(x_j)]
	\end{equation} is invertible. We will use this matrix in the proof of Theorem $\ref{MT}$.
	
	One other result we need in the next section is \cite[\text{Lemma 6.5}]{Isa88}. This states that for any $\vp \in \ibr(G_\t|\t)$, we have
	\begin{equation} \label{Eq}
		\vp(nx_i) = \vp(x_i)\frac{\t_i(nx_i)}{\t_i(x_i)}
	\end{equation}
	for any $x_i \in X$ and $n \in N$ such that $nx_i \in G_\t^0$.
	\section{Real-valued Brauer characters over an irreducible Brauer character of a normal subgroup}
	Let $\ibr_{\R}(G)$ be the set of irreducible real-valued Brauer characters of $G$, and let $\ibr_{\R}(G|\t)$ be the set of irreducible real-valued Brauer characters of $G$ lying over $\t$. Recall our assumption that $\t$ is $G$-conjugate to $\overline{\t}$.
	
	The Clifford correspondence is the bijection $\ibr(G_\t|\t) \to \ibr(G|\t)$, given by induction. Then induction also defines a bijection $\ibr(G_\t^*|\t) \to \ibr(G|\t)$. We refer to this as the \textit{extended Clifford correspondence}. Since $\t$ is $G$-conjugate to $\overline{\t}$, complex conjugation acts on both $\ibr(G_\t^*|\t)$ and $\ibr(G|\t)$. So the extended Clifford correspondence restricts to a bijection $\ibr_\R(G_\t^*|\t) \to \ibr_\R(G|\t)$. For the rest of this section, we can and do assume that $G = G_\t^*$. So $[G:G_\t] \leq 2$, and $G = G_\t$ if and only if $\t = \overline{\t}$.
	
	Recall that we fixed $e \in G$ to be any element that satisfies $\t^e = \overline{\t}$. So if $\t = \overline{\t}$, then $e$ is any element of $G$, while if $\t \neq \overline{\t}$, $e$ is any element of $G - G_\t$.
	
	Before we start proving Theorem \ref{MT}, we observe that if $Nx \in G_\t/N$ is $\t$-good, then so is $N(x^{-1})^{e^{-1}}$. With this in mind, we impose an additional condition on the set $T$ of the previous section. 
	
	\begin{itemize}
		\item  If $i,j \in \{1, \dots, t\}$ are distinct and $\mathrm{Cl}(N(x_i^{-1})^{e^{-1}}) = \mathrm{Cl}(Nx_j)$, then by replacing $x_j$ with $(x_i^{-1})^{e^{-1}}$, and replacing $\t_j \in \ibr(N\< x_j \> | \t)$ with $\overline{\t}_i^{e^{-1}} \in \ibr(N\< (x_i^{-1})^{e^{-1}} \> | \t)$, if necessary, we may assume that $x_j=(x_i^{-1})^{e^{-1}}$ and $\t_j = \overline{\t}_i^{e^{-1}}$.
	\end{itemize}
	
	So throughout the following proof, the invertible matrix $M$ from $(\ref{Mat})$ will be in terms of the elements of a set $T$ that also satisfies the above condition. We will then define the matrix $\overline{M^e}$, and show that there exists a permutation matrix $X$ and a monomial matrix $Y$ such that $\overline{M^e} = XM = MY$.
	\medskip
	
	\textit{\textbf{Proof of Theorem \ref{MT}.}}
	Clearly $|\ibr_{\R}(G|\t)|$ equals the number of $\vp \in \ibr(G_\t|\t)$ such that $\vp = \overline{\vp^e}$. So we count those Brauer characters. Set $\overline{M^e}:=[\overline{\vp_i^e(x_j)}]$. Now $\overline{\vp^e}\in \ibr(G_\t|\t)$ for any $\vp \in \ibr(G_\t|\t)$. As each row of $M$ is also a row of $\overline{M^e}$, there is a unique $t \times t$ permutation matrix $X$ such that $XM = \overline{M^e}$. The number of rows of $M$ fixed by $X$ (given by $\tr X$) equals the number of elements $\vp \in \ibr(G_\t|\t)$ that satisfy $\vp = \overline{\vp^e}$. Hence $\tr X = |\ibr_{\R}(G|\t)|$. 
	
	Now we compute $M^{-1}\overline{M^e}$, which is our matrix $Y$. To do this, we relate the columns of $M$ and $\overline{M^e}$. Suppose $i\neq j$ and $\mathrm{Cl}(Nx_i) = \mathrm{Cl}(N(x_j^{-1})^{e^{-1}})$. Then the $i^\mathrm{th}$ column of $M$ is equal to the $j^\mathrm{th}$ column of $\overline{M^e}$ by our construction of $T$. We note that if, for some $x_i \in T$, $\mathrm{Cl}(Nx_i)= \mathrm{Cl}(N(x_i^{-1})^{e^{-1}})$, the $i^\mathrm{th}$ column of $M$ may not be equal to the $i^\mathrm{th}$ column of $\overline{M^e}$. To address this, we need a lemma:
	\begin{Lemma} \label{L3}
		If $x_i \in T$ satisfies $\mathrm{Cl}(Nx_i)= \mathrm{Cl}(N(x_i^{-1})^{e^{-1}})$, then $\vp(x_i) = \s(Nx_i)\overline{\vp^e(x_i)}$ for all $\vp \in \ibr(G_\t|\t)$.
	\end{Lemma}
	Before we prove Lemma $\ref{L3}$, recall our indicator $\s$ for $\t$-good elements of $G_{\t}/N$ that we defined above. We need one additional result: 
	\begin{Lemma} \label{L4}
		Let $Nx \in G_\t/N$ be $\t$-good, and suppose $\mathrm{Cl}(Nx) = \mathrm{Cl}(N(x^{-1})^{e^{-1}})$. Then for any $\varepsilon \in G$ satisfying $Nx^{\varepsilon^{-1}} = Nx^{-1}$ and $\t^\varepsilon = \overline{\t}$, we have $\t_x^\varepsilon(x) = \s(Nx)\overline{\t}_x(x)$.
	\end{Lemma}
	
	\textit{\textbf{Proof of Lemma \ref{L4}.}} Since $\t_x^\varepsilon$ and $\overline{\t}_x$ lie above $\overline{\t}$, there exists $\l_x \in \ibr(N\< x\>/N)$ such that $\t_x^\varepsilon = \l_x\overline{\t}_x$. Note that $\l_x$ does not depend on the choice of $\t_x$ and $\varepsilon$. Since $\varepsilon^2 \in C_G(Nx)$, $Nx$ is $\t$-good, $Nx^{\varepsilon^{-1}} = Nx^{-1}$ and $\overline{\t_x^\varepsilon} = \overline{\t}_x^\varepsilon$, we have
	$$\t_x = \t_x^{\varepsilon^2} = (\l_x \overline{\t}_x)^\varepsilon = \l_x^\varepsilon \overline{\t}_x^\varepsilon= \overline{\l_x} \overline{\t_x^\varepsilon}= \overline{\l}_x \left(\overline{\l_x \overline{\t}_x}\right)=\overline{\l}_x^2 \t_x.$$ 
	So $\overline{\l}_x^2$ is trivial, whence $\l_x = \overline{\l}_x$. Then $\t_x^\varepsilon = \overline{\t}_x$ if $\l_x(x) = 1$ and $\t_x^\varepsilon \neq \overline{\t}_x$ if $\l_x(x) = -1$. In either case, note that $\s(Nx) = \l_x(x)$. So $\t_x^\varepsilon(x) = \s(Nx)\overline{\t}_x(x)$, as claimed. \hfill\ensuremath{\blacksquare}\medskip
	
	\textit{\textbf{Proof of Lemma \ref{L3}.}} Since we assumed that $\mathrm{Cl}(Nx_i)= \mathrm{Cl}(N(x_i^{-1})^{e^{-1}})$, there exists an $\varepsilon \in G$ that satisfies $Nx_i^{\varepsilon^{-1}} = Nx_i^{-1}$ and $\t^\varepsilon = \overline{\t}$. Thus $\s(Nx_i)\in \{+1,-1\}$.  Now $(x_i^{-1})^{\varepsilon^{-1}} = n_ix_i$ for some $n_i \in N$. So $\t_i^\varepsilon(x_i^{-1}) = \t_i(n_ix_i) = \s(Nx_i)\overline{\t_i(x_i^{-1})} = \s(Nx_i)\t_i(x_i)$. Since none of these are zero, we deduce that $\frac{\t_i(n_ix_i)}{\t_i(x_i)} = \s(Nx_i)$.
	
	Now $(x_i^{-1})^{e^{-1}}$ is $G_{\t}$-conjugate to $(x_i^{-1})^{\varepsilon^{-1}} = n_ix_i$. So for $\vp \in \ibr(G_\t|\t)$, we have $\vp^e(x_i^{-1}) = \vp(n_ix_i) = \vp(x_i)\frac{\t_i(n_ix_i)}{\t_i(x_i)} = \s(Nx_i)\vp(x_i)$. Here the second equality is obtained from equation $(\ref{Eq})$. We conclude that $\vp(x_i) = \s(Nx_i)\overline{\vp^e(x_i)}$, as desired. \hfill\ensuremath{\blacksquare}\medskip

	Lemma $\ref{L3}$ tells us that if $\mathrm{Cl}(Nx_i)= \mathrm{Cl}(N(x_i^{-1})^{e^{-1}})$ and $Nx_i$ has $\s_+$ type, then the $i^\mathrm{th}$ columns of $M$ and $\overline{M^e}$ are equal, and if $Nx_i$ has $\s_-$ type, then the $i^\mathrm{th}$ columns of $M$ and $\overline{M^e}$ differ by a scalar factor of $-1$.
	
	We can now fully describe $M^{-1}\overline{M^e} = Y$. By our previous observations, $Y$ has the following properties:
	\begin{itemize}
		\item If  $\mathrm{Cl}(Nx_i)= \mathrm{Cl}(N(x_i^{-1})^{e^{-1}})$, and $\s(Nx_i)=+1$, then $Y_{ii} = 1$ and $Y_{ij} = Y_{ji} = 0$ for any $j \neq i$.
		\item If $\mathrm{Cl}(Nx_i)= \mathrm{Cl}(N(x_i^{-1})^{e^{-1}})$, and $\s(Nx_i)=-1$, then $Y_{ii} = -1$ and $Y_{ij} = Y_{ji} = 0$ for any $j \neq i$.
		\item If $i \neq j$ and $\mathrm{Cl}(Nx_i) = \mathrm{Cl}(N(x_j^{-1})^{e^{-1}})$, then $Y_{ij} = Y_{ji} = 1$ and $Y_{kj} = Y_{jk} = 0$ for any $k \neq i$.
	\end{itemize}
	
	Note that $Y$ is a monomial matrix. As a consequence of the above properties, $\tr Y$ is equal to the number of $p$-regular $\t$-good conjugacy classes of $G_\t/N$ of $\s_+$ type minus the number of $p$-regular $\t$-good conjugacy classes of $G_\t/N$ of $\s_-$ type. Since $\overline{M^e} = XM = MY$, we have $\tr X = \tr (MYM^{-1}) = \tr Y$. Recall that $\tr X$ is equal to the number of Brauer characters $\vp \in \ibr(G_\t|\t)$ that satisfies $\vp = \overline{\vp^e}$, which is also equal to $|\ibr_{\R}(G|\t)|$. Now our proof of Theorem \ref{MT} is complete. \hfill\ensuremath{\blacksquare}

	\section{Special cases and Examples}
 
	\subsection*{Case (i): $\t = \overline{\t}$}
	
	We would have $G_\t^* = G_\t$, and $e \in G_\t$. Hence the number of real-valued irreducible Brauer characters of $G$ lying over $\t$ is equal to the number of real $p$-regular $\t$-good conjugacy classes of $G_\t/N$ of $\s_+$ type minus the number of real $p$-regular $\t$-good conjugacy classes of $G_\t/N$ of $\s_-$ type. 
	
	As a consequence of Lemma \ref{L3}, if $Nx_i$ is real in $G_\t/N$ and $\s(Nx_i) = +1$, then $\vp(x_i) = \overline{\vp(x_i)}$ for all $\vp \in \ibr(G_\t|\t)$, and hence the $i^\mathrm{th}$ column of $M$ from $(\ref{Mat})$ is real-valued. If $Nx_i$ is real in $G_\t/N$ and $\s(Nx_i) = -1$, then $\vp(x_i) = -\overline{\vp(x_i)}$ for all $\vp \in \ibr(G_\t|\t)$, and hence the $i^\mathrm{th}$ column of $M$ from $(\ref{Mat})$ is purely imaginary.
	
	\begin{Example}
		Let $G = GL(2,3)$ and $N = Q_8$. Let $p$ be any prime that does not divide $|G|$, and let $\t$ be the unique irreducible character of $Q_8$ of degree $2$. Then the elements of $\ibr(G|\t)$ (which equals $\irr(G|\t)$) have character values: \medskip
		
		\begin{tabular}{c|cccccccc}
			& $1a$ & $2a$ & $2b$ & $3a$ & $4a$ & $6a$ & $8a$ & $8b$\\ \hline
			$\vp_1$ & $2$  & $-2$ & $0$  & $-1$  & $0$  & $1$ & $-\sqrt{2}i$ & $\sqrt{2}i$ \\
			$\vp_2$ & $2$  & $-2$ & $0$  & $-1$  & $0$  & $1$  & $\sqrt{2}i$ & $-\sqrt{2}i$\\
			$\vp_3$ & $4$ & $-4$ & $0$  & $1$  & $0$  & $-1$  & $0$ & $0$ 
		\end{tabular}
		\medskip
		
		In this case $G/N \cong S_3$. There is one real-valued irreducible character of $G$ lying above $\t$. In comparison, the trivial conjugacy class of $G/N$ and the unique conjugacy class of $G/N$ containing elements of order $3$ are the two $\s_+$ type conjugacy classes, while the unique conjugacy class of $G/N$ consisting of elements of order $2$ is the only $\s_-$ type conjugacy class.
		
		We can choose $x_1, x_2$ and $x_3$ respectively to be any element from the union of conjugacy classes $K_{1a}\cup K_{2a}$, $K_{3a}\cup K_{6a}$ and $K_{8a}\cup K_{8b}$. We can see that in any case, $x_1$ and $x_2$ are of $\s_+$ type, and $\vp(x_1)$ and $\vp(x_2)$ are real for any $\vp \in \irr(G|\t)$, while $x_3$ is of $\s_-$ type and $\vp(x_3)$ is purely imaginary for any $\vp \in \irr(G|\t)$.
	\end{Example}
	
	\subsection*{Case (ii): $p=2$}
	We have the following result:
	\begin{Proposition}
		Let $p = 2$, $\t \in \ibr(N)$ is an irreducible $2$-Brauer character, and assume that $G= G_\t^*$. Let $e \in G_\t^*$ satisfy $\t^e = \overline{\t}$. Let $T = \{ x_1, \dots, x_t\}$ be defined like before. Then for any $x_i \in T$ that satisfies $\mathrm{Cl}(Nx_i)= \mathrm{Cl}(N(x_i^{-1})^{e^{-1}})$, $\s(Nx_i) = +1$. Thus every $2$-regular $\t$-good conjugacy class of $G_\t/N$ is either of $\s_+$ type or of $\s_0$ type.
	\end{Proposition}  
	\begin{Proof}
		Let $Nx_i$ satisfy $\mathrm{Cl}(Nx_i)= \mathrm{Cl}(N(x_i^{-1})^{e^{-1}})$. Then there exists an $\varepsilon \in G_\t^*$ satisfying $Nx_i^{\varepsilon^{-1}} = Nx_i^{-1}$ and $\t^\varepsilon = \overline{\t}$. Let $\t_i \in \ibr(N\langle x_i\rangle)$ be an extension of $\t$. Then $\t_i^\varepsilon = \l_i \overline{\t}_i$ for some $\l_i \in \ibr(N\langle x_i\rangle/N)$. Since $\varepsilon^2 \in C_G(Nx)$, $Nx_i$ is $\t$-good, $Nx_i^{\varepsilon^{-1}} = Nx_i^{-1}$ and $\overline{\t_i^\varepsilon} = \overline{\t}_i^\varepsilon$, we have
		$$\t_i = \t_i^{\varepsilon^2} = (\l_i \overline{\t}_i)^\varepsilon = \l_i^\varepsilon \overline{\t}_i^\varepsilon= \overline{\l_i} \overline{\t_i^\varepsilon}= \overline{\l}_i \left(\overline{\l_i \overline{\t}_i}\right)=\overline{\l}_i^2 \t_x.$$ 
		So $\overline{\l}_i^2$ is trivial, hence $\l_i = \overline{\l}_i$. But since $Nx_i$ is $p$-regular, $N\langle x_i\rangle/N$ is a $p'$-group, which means $\l_i$ is trivial. Hence $\t_i^\varepsilon = \overline{\t}_i$, and $\s(Nx_i) = +1$.
	\end{Proof}
	
	Additionally, in the case where both $\t = \overline{\t}$ and $p=2$, every real $2$-regular conjugacy class of $G_\t/N$ is $\t$-good (this is one consequence of \cite[\text{Theorem 2}]{GM21}). So in this case the number of real valued irreducible Brauer characters of $G$ lying over $\t$ is equal to the number of real $2$-regular conjugacy classes of $G_\t/N$.
	
	\section{Generalizing Theorem \ref{MT}}
	
	Theorem \ref{MT} can be generalized in the following way. Let $N \trianglelefteq G$, $\t \in \ibr(N)$, and $\mathcal{G} = \gal(\Q_{|G|}/\Q_{|G|_p})$. We let $\mathcal{H}$ be the subgroup of $\mathcal{G}$ consisting of elements $\mu$ such that, for some non-negative integer $f$, $\mu(\xi) = \xi^{p^f}$ for every $p'$-root of unity $\xi$ in $\Q_{|G|}$ (cf. \cite{Nav04}). 
	
	Now let $\mu \in \mathcal{H}$ be an involution. We can then define the $\mu$-extended inertia group of $\t$ in $G$ to be the subgroup
	$$G_\t^\mu := \{ g \in G| \t^g = \t \text{ or } \t^\mu\}$$
	So $G_\t \leq G_\t^\mu$ and $[G_\t^\mu: G_\t] \leq 2$. Moreover $G_\t = G_\t^\mu$ if and only if $\t = \t^\mu$. Fix $e \in G_\t^\mu$ to be any element that satisfies $\t^e = \t^\mu$.
	
	Now let $Nx \in G_\t/N$ be $\t$-good and $p$-regular, and let $\t_x$ be an arbitrary extension of $\t$ to $N\< x \>$. We assign an indicator $\s^\mu(Nx) \in \{0,+1,-1\}$ as follows. Suppose first that $\text{Cl}(Nx) = \text{Cl}(N(x^\mu)^{e^{-1}})$. Then there exists $\varepsilon \in G_\t^*$ such that $\t^\varepsilon = \t^\mu$ and $Nx^{\varepsilon^{-1}} = Nx^\mu$. Set $\s^\mu(Nx) = +1$, if $\t_x^\varepsilon = \t_x^\mu$, and $\s^\mu(Nx) = -1$, if $\t_x^\varepsilon \neq \t_x^\mu$. If $\text{Cl}(Nx) \neq \text{Cl}(N(x^{\mu})^{e^{-1}})$ then no such $\varepsilon$ exists, and we set $\s^\mu(Nx) = 0$. 
	
	It can be shown that $\s^\mu(Nx)$ does not depend on the choice of $\varepsilon$, nor on the choice of extension $\t_x$. Additionally, $\s^\mu(Nx) = \s^\mu(Nx^g)$ for any $g \in G_\t$. So $\s^\mu$ is constant on conjugacy classes of $G_\t/N$. Thus we can refer to each $\t$-good conjugacy class of $G_\t/N$ as being of $\s^\mu_+$, $\s^\mu_-$ or $\s^\mu_0$ type.
	
	We define $\ibr_\mu(G|\t) := \{ \vp \in \ibr(G|\t)| \vp = \vp^\mu\}$. Now this set is empty unless $\t$ is $G$-conjugate to $\t^\mu$, so we may assume that $\t$ is $G$-conjugate to $\t^\mu$. Then $\mu$ has an equivalent version of Theorem \ref{MT}:
	
	\begin{Theorem}
		Let $N \trianglelefteq G$ and let $\t \in \ibr(N)$. Then $|\ibr_\mu(G|\t)|$ is equal to the number of $p$-regular $\t$-good conjugacy classes of $G_\t/N$ of $\s^\mu_+$ type minus the number of $p$-regular $\t$-good conjugacy classes of $G_\t/N$ of $\s^\mu_-$ type.
	\end{Theorem}
	
	We omit the proof of this more general theorem, since it can be proven in a similar way to Theorem \ref{MT}.
	\section*{Acknowledgements}
	This paper was written with the support of the John \& Pat Hume Doctoral Awards financed by Maynooth University. I also acknowledge the encouragement and advice of my PhD advisor John Murray.
	
\end{document}